\documentclass[12pt,twoside,a4paper]{article}   
\usepackage[utf8]{inputenc} 
\usepackage{amsfonts}
\usepackage{amssymb}
\ifx\pdfpageheight\undefined
   \usepackage[dvips,colorlinks=true,linkcolor=blue,citecolor=red,%
      urlcolor=green]{hyperref}
   \usepackage[dvips]{graphicx}
   \makeatletter
   \edef\Gin@extensions{\Gin@extensions,.mps}
   \makeatother
\else
   \usepackage[pdftex]{graphicx}
   \usepackage[bookmarksopen=false,pdftex=true,breaklinks=true,%
      backref=page,pagebackref=true,plainpages=false,%
      hyperindex=true,pdfstartview=FitH,colorlinks=true,%
      pdfpagelabels=true,colorlinks=true,linkcolor=blue,%
      citecolor=red,urlcolor=red,hypertexnames=false%
      ]%
   {hyperref}
\fi

\usepackage{amsfonts,amssymb}
\usepackage[french]{babel}
\usepackage[all]{xy}
\usepackage{mathrsfs}
\usepackage{lmodern}
\usepackage{microtype}
\usepackage{array}
\usepackage{stmaryrd}
\usepackage{wasysym}
\usepackage{xspace}

\FrenchFootnotes

\newtheorem{theorem}{Théorème}

\newtheorem{lemma}[theorem]{Lemme}%[section]
\newtheorem{definition}[theorem]{Définition}%[section]
\newtheorem{proposition}[theorem]{Proposition}%[section]

\newenvironment{proof}{
\trivlist \item[\hskip \labelsep{\it Démonstration.}]}{\hfill\mbox{$\Box$}
\endtrivlist}

\newcommand\oge{\leavevmode\raise.3ex\hbox{$\scriptscriptstyle\langle\!\langle\,$}}
\newcommand\feg{\leavevmode\raise.3ex\hbox{$\scriptscriptstyle\,\rangle\!\rangle$}}
\newcommand\gui[1]{\oge{#1}\feg}

\newcommand \aqo[2]{#1\sur{\gen{#2}}\!}

\newcommand \gen[1]{\left\langle{#1}\right\rangle}
\newcommand \geN[1]{\big\langle{#1}\big\rangle}
\newcommand \ov[1] {\overline{#1}}
\newcommand \sur[1]{\!\left/#1\right.}

\newcommand \lst[1] {\hbox{$[\,#1\,]$}}

\newcommand\scp[2] {\gen{#1\,|\,#2}\!}

\renewcommand\matrix[1]{{\begin{array}{ccccccccccccccccccccccccc} #1 \end{array}}}  
\newcommand\cmatrix[1]{\left[\matrix{#1}\right]}
\newcommand\Cmatrix[2]{\setlength{\arraycolsep}{#1}\left[\matrix{#2}\right]}
\newcommand\crmatrix[1]{{\left[\begin{array}{rrrrrrr} #1 \end{array}\right]}}

\newcommand\lrb[1] {\llbracket #1 \rrbracket}
\newcommand\lrbn {\lrb{1..n}}

\newcommand\tra[1]{{\,^{\rm t}\!#1}}

\newcommand\Ae[1]{\gA^{\!#1}}

\newcommand \som {\sum\nolimits}
\newcommand \ds {\displaystyle}

\newcommand \snic[1]{\mni

{\centering$#1$\par}

}

\newcommand \rem{
\noindent {\it Remarque. }}

\newcommand \eoe {\hbox{}\nobreak\hfill
\vrule width 1.4mm height 1.4mm depth 0mm \par \smallskip}

\newcommand \lora {\longrightarrow}
\newcommand{\llongrightarrow}{\relbar\joinrel\mkern-1mu\longrightarrow}
\newcommand{\lllongrightarrow}{\relbar\joinrel\mkern-1mu\llongrightarrow}
\newcommand{\llllongrightarrow}{\relbar\joinrel\mkern-1mu\lllongrightarrow}
\newcommand \llra {\llongrightarrow}

\newcommand \llllra {\llllongrightarrow}

\newcommand\vvers[1]{\buildrel{#1}\over \llra }

\newcommand\vvvvers[1]{\buildrel{#1}\over \llllra }

\newcommand \sni {\smallskip \noindent }
\newcommand \mni {\medskip \noindent }

\newcommand \MM{\mathbb{M}}

\newcommand \Mn{{\MM_n}}
\newcommand \GL{\mathbb{GL}}
\newcommand \GLn{{\GL_n}}

\newcommand \ua  {{\underline{a}}}

\newcommand \ug  {{\underline{g}}}
\newcommand \uu  {{\underline{u}}}

\newcommand \uf  {{\underline{f}}}
\newcommand \uX  {{\underline{X}}}

\renewcommand \Im {\mathrm{Im}}

\newcommand \fa{\mathfrak{a}}
\newcommand \fc{\mathfrak{c}}

\newcommand \gk{\mathbf{k}}
\newcommand \gA{\mathbf{A}}
\newcommand \gB{\mathbf{B}}

\newcommand \an {a_1,\ldots,a_n}
\newcommand \bn {b_1,\ldots,b_n}
\newcommand \gn {g_1,\ldots,g_n}

\newcommand \xn {x_1,\ldots,x_n}
\newcommand \Xn {X_1,\ldots,X_n}
\newcommand \lfs {f_1,\ldots,f_s}

\newcommand \kuX {{\gk[\uX]}}

\newcommand \kXn {{\gk[\Xn]}}
\newcommand \kxn {{\gk[\xn]}}

\newcommand \klg {$\gk$-\alg}

\newcommand \Alg {$\gA$-\alg}

\newcommand \Amo {$\gA$-module\xspace}

\newcommand \kmo {$\gk$-module\xspace}

\newcommand \ssi {si, et seulement si, }

\newcommand \alg {algèbre\xspace}

\newcommand \bil {bilinéaire\xspace}

\newcommand \carn{caractérisation\xspace}

\newcommand \coes {coefficients\xspace}

\newcommand \csce {complètement sécante\xspace}

\newcommand \csces {complètement sécantes\xspace}

\newcommand \dfn{définition\xspace}

\newcommand \eds {extension des scalaires\xspace}

\newcommand \fit {fidèlement\xspace}

\newcommand \fpte {\fit plate\xspace}

\newcommand \ftm {fortement\xspace}

\newcommand \fuc {\ftm \usc}  %fortement $1$-sécante 

\newcommand \id {idéal\xspace}

\newcommand \itf {\id \tf}

\newcommand \iv {inversible\xspace}

\newcommand \mptf {module \ptf}

\newcommand \ndz {régulier\xspace}

\newcommand \ndze {régulière\xspace}
\newcommand \rege {régulière\xspace}

\newcommand \pf {de \pn finie\xspace}

\newcommand \pn {présentation\xspace}

\newcommand \pol {polynôme\xspace}
\newcommand \pols {polynômes\xspace}

\newcommand \pro {projectif\xspace}

\newcommand \ptf {\pro \tf}

\newcommand \scs {suite \csce}

\newcommand \sfuc {suite \fuc} %fortement $1$-sécante

\newcommand \srg {suite régulière\xspace}

\newcommand \suc {suite \usc}

\newcommand \syzy {syzygie\xspace}
\newcommand \syzys {syzygies\xspace}

\newcommand \tf {de type fini\xspace}

\newcommand \Tho {Théorème }

\newcommand \tho {théorème\xspace}

\newcommand \unt {unitaire }

\newcommand \usc {$1$-sécante\xspace}
\newcommand \uscs {$1$-sécantes\xspace}

\begin{document} 

\title{\Tho de de Smit et Lenstra,\\
démonstration élémentaire}
\author{H. Lombardi, C. Quitté}

\vspace{-2em}

\maketitle



%:     abstract
\begin{abstract} 
Nous donnons une démonstration élémentaire et constructive d'un \tho de de Smit et Lenstra. 

\smallskip  \centerline{\bf Abstract}

\smallskip  
We give  an elementary and constructive proof for a theorem of de Smit et Lenstra.

\end{abstract}

\noindent \emph{Note.} Dans la version 1 sur arXiv, il manquait le démonstration de la proposition \ref{prop-compsec-1sec}. 

\subsection*{Introduction} 
Le \tho de de Smit et Lenstra en question donné dans l'article \cite{SmLe} est le suivant.
\setcounter{theorem}{11}

%:     Theorem{thdeSmitLenstra}
\begin{theorem} \label{thdeSmitLenstra} 
 \emph{(Lenstra {\&} de Smit, une platitude remarquable)}
\\
Soit $\gk$ un anneau arbitraire, et soit une \klg $\gA=\aqo\kXn{f_1,\dots,f_n}$. Si~$\gA$ est finie sur $\gk$, la suite 
$(f_1,\dots,f_n)$ est \csce, l'\alg est plate et c'est un \kmo \ptf.
\end{theorem}

Nous en donnons une démonstration constructive élémentaire en nous appuyant sur l'ouvrage \cite{CACM}.

\setcounter{theorem}{0}

\subsection*{Suites régulières, $1$-sécantes, \csces}
%:  Definition{defSeqReg}------
\begin{definition}
\label{defSeqReg} ~
Une suite $(a_1, \ldots, a_k)$ dans  un anneau $\gA$
est
\emph{régulière}
si chaque $a_i$ est \ndz dans l'anneau
$\aqo{\gA}{a_j\,;\,j<i}$.%
\end{definition}
%--- end-definition------------------------------------

\rem Nous avons retenu ici la \dfn de Bourbaki. La plupart des auteurs
réclament en outre que l'\id $\gen{a_1,\ldots ,a_k}$ ne contienne pas $1$. 
L'expérience montre que cette  négation ne fait qu'introduire des complications et des énoncés tordus.
\eoe

\medskip Étant donnée une suite $(a_1, \ldots, a_k)$ dans  un anneau $\gA$, parmi les \syzys entre les $a_i$ figurent ce que l'on  appelle les \textsl{syzygies triviales}: 
%(ou \textsl{relateurs triviaux} si on les voit
%comme des \rdes \agqs sur $\gk$ lorsque $\gA$ est une \klg): 
\[
{a_ia_j-a_ja_i=0\;\hbox{  pour  }\;i\neq j.}
\]
On notera $R_\ua$  \gui{la} \textsl{matrice des \syzys triviales} (l'ordre des colonnes est sans importance), de format
${n\times  n(n-1)/2}$. Par exemple, pour $n = 5$%
\[
R_{\ua} = \Cmatrix{.3em} {
a_2 & a_3  &  0  & a_4  &  0  &   0 & a_5 & 0 & 0 & 0\cr
-a_1&   0 &  a_3 &   0 &  a_4 &   0& 0 & a_5 & 0 & 0\cr
 0 & -a_1 & -a_2 &   0 &   0 &  a_4&0  &   0 & a_5&   0\cr
 0 &   0 &   0 & -a_1 & -a_2 & -a_3&  0  &   0 & 0& a_5\cr
 0 &   0 &   0 & 0 &   0 &   0 & -a_1 & -a_2 & -a_3& -a_4}. 
\]

%:     Definition defiCseqFFR  
\begin{definition} \label{defi1seqFFR}
Soit $(\ua)=(\an)$ dans $\gA$. 
On dit que la suite $(\ua)$ est \emph{\usc} si le module des $\gA$-\syzys entre les~$a_i$ est engendré par les \syzys triviales.
Autrement dit, la suite $(\an)$ est \usc \ssi
la suite 
\[
\Ae {n(n-1)/2} \vvers {R_\ua}  \Ae n
\vvvvers {[a_1\cdots a_n]}  \gen{\an} \lora 0.
\]
est exacte. En particulier l'\id $\gen{\an}$ est un \Amo \pf.
\end{definition}
%--------- fin definition -----------------------------------------

Rappelons qu'une matrice $M = (m_{ij}) \in \Mn(\gA)$ est dite
\textsl {alternée} si c'est la matrice d'une forme \bil alternée,
i.e. $m_{ii} = 0$ et $m_{ij} + m_{ji} = 0$ pour~$i$,~$j\in\lrbn$.%
\index{matrice!alternée}\index{alternée!matrice ---}

Le \Amo des matrices alternées
 est libre de rang  ${n(n-1)}\over 2$ et admet une base naturelle.
Par exemple, pour $n = 3$,
$$
\Cmatrix{.3em} {0 & a & b \cr -a & 0 & c \cr -b & -c & 0 \cr} =
a\cmatrix {0 & 1 &  0 \cr -1 & 0 & 0 \cr 0 & 0 & 0 \cr} +
b\cmatrix {0 & 0 & 1 \cr 0 & 0 & 0 \cr -1 & 0 & 0 \cr} +
c\cmatrix {0 & 0 & 0 \cr 0 & 0 & 1 \cr 0 & -1 & 0 \cr}.
$$

%%:    lemma{PetitLemmeAlterne}
\begin {lemma}\label{PetitLemmeAlterne}
Soit $a = \tra{[\,\ua\,]} = \tra{[\,a_1\;\cdots\;a_n\,]} \in \Ae {n\times 1}$.
\begin {enumerate}
\item
Soit $M  \in \Mn(\gA)$ une matrice alternée;
on~a $\scp {Ma}{a} = 0$.

\item
Un $u \in \Ae {n\times 1}$ est dans $\Im R_\ua$ \ssi il
existe une matrice alternée $M\in \Mn(\gA)$
telle que $u = Ma$.

\item La suite $(\an)$ est \usc \ssi  pour tout $u \in \Ae {n\times 1}$ tel que $\scp a u = 0$ il existe une matrice alternée $M$ telle que $\tra {[\,\uu\,]}=M\tra {[\,\ua\,]}$. 
\end {enumerate}
\end {lemma}
%%%%%%%%%%%%%%%%%%%%%%%%%%%%%%%%%%%%%%%%%
\begin {proof}\textsl{1.} En effet, $\scp {Ma}{a} = \varphi(a,a)$, où $\varphi$
est une forme \bil alternée.

\noindent \textsl{2.} Par exemple, pour la première colonne de $R_\ua$ avec $n=4$, on~a:
$$ 
\cmatrix {0 & 1 & 0 & 0\cr -1 & 0 & 0 & 0\cr 0 & 0 & 0 & 0\cr 0 & 0 & 0 & 0\cr}
\crmatrix {a_1\cr a_2\cr a_3\cr a_4\cr} =
\Cmatrix{.2em} {a_2\cr -a_1\cr 0\cr 0\cr}, 
$$
et les ${n(n-1)}\over 2$ colonnes de $R_\ua$ correspondent ainsi
aux ${n(n-1)}\over 2$ matrices alternées formant la base naturelle
du \Amo des matrices alternées de~$\Mn(\gA)$.

\noindent \textsl{3}. Résulte de \textsl{2}.
\end {proof}
%%%%%%%%%%%%%%%%%%%%%%%%%%%%%%%%%%%%%%%%%

%%:     lemma{lemsucpf}
\begin {lemma}\label{lemsucpf} Soit $(\an)$ une \suc dans $\gA$, $V\in\GLn(\gA)$ et 
\[[\,b_1\;\cdots\;b_n]=[\,a_1\;\cdots\;a_n]V
.\] 
Alors $(\bn)$ est \usc.
\end {lemma}
\begin{proof} Considérons une \syzy $[\,b_1\;\cdots\;b_n]\tra{[\,u_1\;\cdots\;u_n]}=0$, i.e. $[\,a_1\;\cdots\;a_n]V\tra{[\,u_1\;\cdots\;u_n]}=0$.
D'après le lemme \ref{PetitLemmeAlterne} il existe une matrice alternée $M$ telle que $V\tra{[\,u_1\;\cdots\;u_n]}=M\tra{[\,a_1\;\cdots\;a_n]}$. Donc 
\[
\tra{[\,u_1\;\cdots\;u_n]}=(V^{-1}M\tra{V^{-1}})\tra{[\,b_1\;\cdots\;b_n]}.
\] 
La matrice $V^{-1}M\tra{V^{-1}}$ est alternée, on conclut avec le lemme \ref{PetitLemmeAlterne}.
\end{proof}

%:     Proposition{prop1sEregpfSL}
\begin{proposition}\label{prop1sEregpfSL}
Toute \srg est \usc.
\end{proposition}
\begin{proof}
Voir \cite[IV-2.5]{CACM}. 
\end{proof}
%

%:     Proposition{prop-ChgBasSrg}
\begin{proposition}[\eds et \srg] \label{prop-ChgBasSrg}~\\
Soient $\rho\colon \gA\to\gB$ une \Alg, $(\an)$ une suite dans~$\gA$.
Notons $(\bn)=\big(\rho(a_1),\dots,\rho(a_n)\big)$.
\begin{enumerate}
\item Si $\gB$ est plate sur $\gA$ et si la suite $(\an)$ est régulière, alors la suite $(\bn)$ est \rege. 
\item Si $\gB$ est \fpte sur $\gA$ et si  la suite $(\bn)$ est \rege, alors la suite $(\an)$ est  \rege.
\end{enumerate} 
\end{proposition}
%--------- fin proposition ------------------------- 

%
\begin{proof}
En effet le fait que la suite est \rege signifie que l'on~a des suites exactes 
\[ 
\begin{array}{cccccccccccc} 
 0&  \to & \gA   & \vvers{\cdot\times a_ 1}   & \gA     \\[.1mm] 
 0&  \to &  \gA/ a_ 1\gA  & \vvers{\cdot\times {a_ 2}}   &  \gA/ a_ 1\gA  \\[.1mm] 
 0&  \to &  \gA  / (a_ 1\gA+a_ 2\gA) & \vvers{\cdot\times {a_ 3}}   &\gA  /(a_ 1\gA+a_ 2\gA)    
\\[.1mm] 
&&\vdots&&\vdots
 \end{array}
\]
On applique donc les résultats concernant la platitude (ou fidèle platitude) et les suites exactes.
\end{proof}
%

%:     Proposition{prop-ChgBasSuc}
\begin{proposition}[\eds et suite \usc] \label{prop-ChgBasSuc}~\\
Soient $\rho\colon \gA\to\gB$ une \Alg, $(\an)$ une suite dans~$\gA$.
Notons $(\bn)=\big(\rho(a_1),\dots,\rho(a_n)\big)$.
\begin{enumerate}
\item Si $\gB$ est plate sur $\gA$ et si la suite $(\an)$ est \usc, alors la suite $(\bn)$ est \usc. 
\item Si $\gB$ est \fpte sur $\gA$ et si  la suite $(\bn)$ est \usc, alors la suite $(\an)$ est  \usc.
\end{enumerate} 
\end{proposition}
%--------- fin proposition ------------------------- 
%
\begin{proof} Cela résulte de la \carn des suites \uscs par l'exactitude d'une suite (\dfn \ref{defi1seqFFR}).
\end{proof}
%
%:     Lemma{lem-compsec-1}
\begin{lemma} \label{lem-compsec-1}
Dans un anneau $\gA$, considérons un \itf $\fa=\gen{\an}$
et une \srg $(\bn)$ dans $\fa$. Alors dans l'extension de Nagata  \fpte $\gB=\gA(\Xn)$, l'\id $\fa$ est engendré par une \srg $(g_1,\dots,g_n)$. En outre la matrice de passage de $(\ua)$ à $(\ug)$ est \iv. 
\end{lemma}
%----------- fin lemma ----------------------------------- 
%
\begin{proof}
Voir \cite[point 1 de la proposition B-2.4]{RLF}. 
\end{proof}
%

%d
%:     Definition{defi-compsec}
\begin{definition} \label{defi-compsec}
Soit $(\ua)=(\an)$ dans $\gA$. 
On dit que la suite $(\ua)$ est \textsl{\csce} s'il existe une \Alg \fpte $\varphi:\gA\to\gB$ telle que l'\id $ \gen{\varphi(a_1),\dots,\varphi(a_n)} $ contient une suite $(\bn)$ \ndze. 
\end{definition}
%----------- fin definition -------------------------------- 

%p
%:     Proposition{prop-compsec-1sec}
\begin{proposition} \label{prop-compsec-1sec}
Toute suite \csce est \usc. 
\end{proposition}
%----------- fin proposition ----------------------------- 
%
\begin{proof}
Soit $(\an)$ la \scs dans $\gA$. D'après le lemme précédent, dans une extension \fpte $\gB$ de $\gA$, l'\id $\fa=\gen{\an}$ est engendré par une \srg $(\gn)$, qui est \usc d'après la proposition \ref{prop1sEregpfSL}. 
En outre la matrice de passage de $(\ua)$ à $(\ug)$ est \iv. Il s'ensuit que $(\an)$ est \usc dans $\gB$ (lemme \ref{lemsucpf}), puis dans $\gA$ par fidèle platitude (proposition \ref{prop-ChgBasSuc}).
\end{proof}

\subsection*{Suites \uscs et platitude}

%:     Definition{defisfuc}
\begin{definition} \label{defisfuc} Soit $\gk$ un anneau arbitraire.
On dit que $(f_1,\dots,f_s)$ dans $\kXn$ est une \emph{\sfuc} si pour tout \itf
$\fc$ de $\gk$, la suite $(\ov{f_1},\dots,\ov{f_n})$ est \usc dans $(\gk/\fc)[\Xn]$.%
\end{definition}
%--------- fin definition ----------------------------------------------

%:     Theorem{thsfuc}
\begin{theorem} \label{thsfuc}
\emph{(Relations triviales et platitude)}\\
Soit $\gk$ un anneau arbitraire, et soit $(\lfs)$  une \sfuc sur $\kuX$. Alors l'\alg quotient $\gA=\aqo\kuX{\lfs}$ est plate sur $\gk$.
\end{theorem}
%--------- fin theorem ---------------------------------------------- 
%
\begin{proof}
D'après le critère de platitude pour le quotient d'un $\gk$-module plat
(\cite[proposition VIII-1.16]{CACM}), ici le
quotient du \kmo $P=\kuX$ par le \hbox{sous-\kmo $\geN{\uf}$,  $\gA$} est plate sur~$\gk$ \ssi pour tout  \itf~$\fa$ de $\gk$, on a l'inclusion $\geN{\uf}\cap \fa P\subseteq \fa \geN{\uf}$, i.e.

\snic {
\gen {\lfs} \cap \fa[\uX] \subseteq  \fa f_1 + \cdots + \fa f_s.
}

\mni
C'est alors vrai pour tout \id $\fa$ de $\gk$, \tf ou non.
Concrètement, cela signifie que si $h=u_1f_1 + \cdots + u_sf_s\in\fa[\uX]$  (avec \hbox{les $u_i \in \gk[\uX]$}), on peut  réécrire $h$ sous la forme~\hbox{$v_1f_1 + \cdots + v_sf_s$} avec les~$v_i\in\fa[\uX]$.

\sni
Soit $h=u_1f_1 + \cdots + u_sf_s\in\fa[X]$; passons au quotient $(\gk/\fa)[\uX]$:

\snic {
\ov h =\ov{u_1}\,\ov{f_1} + \cdots +\ov{u_s}\,\ov{f_s} = 0.
}

\sni
Comme la suite $(\ov{f_1}, \ldots, \ov{f_s})$ est \usc, il existe
une matrice alternée $s \times s$,  à \coes dans $(\gk/\fa)[\uX]$,
notons là $N$, telle que
\[
\lst{\ov{u_1}\; \cdots\; \ov{u_s}} = \lst{\ov{f_1}\; \cdots\; \ov{f_s}}\, N
\]
On relève $N$ en une matrice alternée $M \in \MM_s(\gk[\uX])$ et
on définit des \pols $w_i \in \kuX$ par:
\[
\lst{w_1\; \cdots\; w_s} = \lst{f_1\; \cdots\; f_s}\, M.
\]
de sorte que, d'une part $\sum_i w_if_i = 0$ et d'autre part $\ov {w_i} = \ov {u_i}$
modulo $\fa$. On a alors:

\snic {\ds
\som_i u_i f_i = \som_i (u_i - w_i)f_i  \quad
\hbox {avec  }u_i - w_i  \hbox { à \coes dans } \fa.
}
\end{proof}

\subsection*{Une platitude remarquable}

%:     lemma
\begin{lemma} \label{lemCompSec1}
Soit $\gk$ un anneau arbitraire,~$\fa$ un \id de $\gA=\kXn$ et $\gB=\gA/\fa=\kxn$. On suppose que~$\gB$ est finie sur $\gk$. 
Alors $\fa$ contient une \srg de longueur $n$.
\end{lemma}
\begin{proof} Pour chaque $i$, soit $f_i\in\gk[X_i]$ \unt qui annule $x_i\in\gB$.\\
On montre que la suite $(f_n,\dots,f_{1})$ est une \srg de $\gA$.\\
On pose $\gA_0=\gA[X_1,\dots,X_{n-1}]$. Le \pol $f_{n}$ est \unt en~$X_n$, donc \ndz dans $\gA=\gA_0[X_n]$. On considère le quotient 
\[
\gB_1=\aqo\gA{f_{n}}=\gk[X_{1},\dots,X_{n-1},x_n]=\gA_1[X_{n-1}]
\]
avec $\gA_1=\gA[X_1,\dots,X_{n-2},x_{n}]$. Le \pol $f_{n-1}$ est \unt en~$X_{n-1}$  donc \ndz dans $\gA_1[X_{n-1}]$.\\
Et ainsi de suite.
\end{proof}
%

%:     Theorem{thdeSmitLenstra}
\begin{theorem} \label{thdeSmitLenstra} 
 \emph{(Lenstra {\&} de Smit, une platitude remarquable. \cite{SmLe})}
\\
Soit $\gk$ un anneau arbitraire, et soit une \klg $\gA=\aqo\kXn{f_1,\dots,f_n}$. Si~$\gA$ est finie sur $\gk$, la suite 
$(f_1,\dots,f_n)$ est \csce, l'\alg est plate et c'est un \kmo \ptf.
\end{theorem}
%--------- fin theorem ---------------------------------------------- 
%
\begin{proof}
La \klg $\gA$ est \pf, et finie comme \kmo, elle est donc \pf comme
\kmo (\cite[\tho VI-3.17]{CACM}). 
Si elle est plate, c'est un $\gk$-\mptf.
\\
On montre que $(f_1,\dots,f_n)$ est une \sfuc. 
 En fait pour tout \id~$\fc$
de $\gk$, la suite $(\ov{f_1},\dots,\ov{f_n})$ est \csce dans~$(\gk/\fc)[\Xn]$, car l'\id $\geN{\ov{f_1},\dots,\ov{f_n}}$ contient une \srg de longueur $n$ comme démontré dans le lemme \ref{lemCompSec1}.\\
D'après le \tho \ref{thsfuc} l'\alg est plate.  
 \end{proof}

\end {document}